\documentclass[11pt,leqno]{amsart}
\usepackage{verbatim,amssymb,amsfonts,latexsym,amsmath,amsthm,bbm,nicefrac,xspace,enumerate,tikz}

\newtheorem{theorem}{Theorem}[section]

\newtheorem{lemma}[theorem]{Lemma}

\newtheorem{question}[theorem]{Question}

\newtheorem*{question*}{Question}
\theoremstyle{definition}

\theoremstyle{remark}
\newtheorem{remark}[theorem]{Remark}
\newcommand{\R}{\mathbb{R}}

\newcommand{\B}{\mathcal{B}}
\newcommand{\C}{\mathcal{C}}

\newcommand{\E}{\mathcal{E}}
\newcommand{\F}{\mathcal{F}}
\newcommand{\G}{\mathcal{G}}

\newcommand{\I}{\mathcal{I}}
\newcommand{\J}{\mathcal{J}}

\newcommand{\M}{\mathcal{M}}
\renewcommand{\P}{\mathcal{P}}

\newcommand{\explicitSet}[1]{\left\lbrace #1 \right\rbrace}
\newcommand{\brackets}[1]{\left\langle #1 \right\rangle}
\newcommand{\set}[2]{\explicitSet{#1 \colon #2}}
\newcommand{\seq}[2]{\brackets{#1 \colon #2}}
\newcommand{\<}{\langle}
\renewcommand{\>}{\rangle}
\renewcommand{\a}{\alpha}

\newcommand{\dlt}{\delta}

\renewcommand{\k}{\kappa}
\newcommand{\s}{\sigma}
\renewcommand{\t}{\tau}
\newcommand{\w}{\omega}
\newcommand{\0}{\emptyset}

\newcommand{\sub}{\subseteq}
\newcommand{\rest}{\!\restriction\!}
\newcommand{\cat}{\!\,^{\frown}}

\newcommand{\homeo}{\approx}

\newcommand{\closure}[1]{\overline{#1}}

\newcommand{\cf}{\mathrm{cf}}

\newcommand{\tr}[1]{[\![#1]\!]}

\newcommand{\BB}{\mathbb{B}}

\newcommand{\continuum}{\mathfrak{c}}

\newcommand{\bld}[1]{\ensuremath{\mathrm {#1}}}
\newcommand{\scr}[1]{\ensuremath{\mathcal {#1}}}
\newcommand{\add}[1]{\ensuremath{\bld{add}(\scr{#1})}}
\newcommand{\non}[1]{\ensuremath{\bld{non}(\scr{#1})}}
\newcommand{\cov}[1]{\ensuremath{\bld{cov}(\scr{#1})}}
\newcommand{\cof}[1]{\ensuremath{\bld{cof}(\scr{#1})}}

\newcommand{\gch}{\ensuremath{\mathsf{GCH}}\xspace}
\newcommand{\zfc}{\ensuremath{\mathsf{ZFC}}\xspace}

\newcommand{\nwd}{\text{\scalebox{.9}{$\mathcal{NWD}$}}}
\newcommand{\HIT}{\scalebox{1}{{\small H}{\scriptsize IT}} }
\newcommand{\MISS}{\scalebox{1}{{\small M}{\scriptsize ISS}} }

\begin{document}

\title[Cardinal invariants of a meager ideal]{Cardinal invariants of a meager ideal}
\author{Will Brian}
\address {
W. R. Brian\\
Department of Mathematics and Statistics\\
University of North Carolina at Charlotte\\
9201 University City Blvd.\\
Charlotte, NC 28223, USA}
\email{wbrian.math@gmail.com}
\urladdr{wrbrian.wordpress.com}


\subjclass[2010]{subjects}

\thanks{The author is supported by NSF grant DMS-2154229.}

\begin{abstract}
Let $\M_X$ denote the ideal of meager subsets of a topological space $X$. 
We prove that if $X$ is a completely metrizable space without isolated points, then the smallest cardinality of a non-meager subset of $X$, denoted $\mathrm{non}(\mathcal M_X)$, is exactly $\mathrm{non}(\mathcal M_X) = \mathrm{cf}[\kappa]^\omega \cdot \mathrm{non}(\mathcal M_\R)$, where $\kappa$ is the minimum weight of a nonempty open subset of $X$. 
We also characterize the additivity and covering numbers for $\M_X$ in terms of simple topological properties of $X$. Some bounds are proved and some questions raised concerning the cofinality of $\M_X$ and the cofinality of the related ideal of nowhere dense subsets of $X$.
 
We also show that if $X$ is a compact Hausdorff space with $\pi$-weight $\kappa$, then $\mathrm{non}(\mathcal M_X) \leq \mathrm{cf}[\kappa]^\omega \cdot \mathrm{non}(\mathcal M_\R)$. This bound for compact Hausdorff spaces is not sharp, in the sense that it is consistent for such a space to have non-meager subsets of even smaller cardinality. 
\end{abstract}

\maketitle


\section{Introduction}

An \emph{ideal} on a set $X$ is a collection $\I$ of subsets of $X$ such that
\begin{enumerate}
\item every finite subset of $X$ is in $\I$, 
\item if $A \in \I$ and $B \sub A$ then $B \in \I$, 
\item if $\J \sub \I$ and $\J$ is finite, then $\bigcup \J \in \I$.
\end{enumerate}
An ideal is \emph{proper} if $\I \neq \P(X)$ or, equivalently, if $X \notin \I$. An ideal $\I$ is called a \emph{$\s$-ideal} if it satisfies the following strengthening of $(3)$:
\begin{itemize}
\item[$(3')$] if $\J \sub \I$ and $\J$ is countable, then $\bigcup \J \in \I$.
\end{itemize}


\noindent Four cardinal numbers are naturally associated to every ideal:

\begin{itemize}
\item[$\circ$] $\mathrm{add}(\I) = \min \{ |\J| :\, \J \sub \I \text{ and } \bigcup \J \notin \I \}$.
\item[$\circ$] $\mathrm{cov}(\I) = \min \{ |\J| :\, \J \sub \I \text{ and } \bigcup \J = X \}$.
\item[$\circ$] $\mathrm{non}(\I) = \min \{ |Y| :\, Y \sub X \text{ and } Y \notin \I \}$.
\item[$\circ$] $\mathrm{cof}(\I) = \min \{ |\J| :\, \J \sub \I \text{ and } \forall A \in \I \,\, \exists B \in \J \,\, B \supseteq A \}$.
\end{itemize}

\noindent These four cardinals are known, respectively, as the \emph{additivity}, the \emph{covering number}, the \emph{uniformity number}, and the \emph{cofinality} of $\I$.
If $\I$ is a proper $\s$-ideal, then these four cardinals satisfy the inequalities displayed below.
\vspace{1mm}

\begin{center}
\begin{tikzpicture}[scale=1]

\node at (-1.8,0) {$\aleph_1$};
\node at (5.8,0) {$|\I|$};
\node at (7.7,0) {$|\mathcal P(X)|$};
\node at (0,0) {$\mathrm{add}(\mathcal I)$};
\node at (2,1) {$\mathrm{non}(\mathcal I)$};
\node at (2,-1) {$\mathrm{cov}(\mathcal I)$};
\node at (4,0) {$\mathrm{cof}(\mathcal I)$};

\node at (-1.1,0) {$\leq$};
\node at (5,0) {$\leq$};
\node at (6.6,0) {$\leq$};
\node at (.9,.55) {\rotatebox{30}{$\leq$}};
\node at (.9,-.55) {\rotatebox{-30}{$\leq$}};
\node at (3.1,.55) {\rotatebox{-30}{$\leq$}};
\node at (3.1,-.55) {\rotatebox{30}{$\leq$}};

\end{tikzpicture}
\end{center}

\vspace{2mm}

If $X$ is a $T_1$ space with no isolated points, then 
$$\M_X = \set{A \sub X}{A \text{ is meager in }X}$$ 
is a $\s$-ideal on $X$. (If $X$ has an isolated point, then $\M_X$ fails requirement $(1)$ in the definition of an ideal.) 
This is a proper $\s$-ideal if and only if $X$ is not a countable union of meager sets. A \emph{Baire space} is a space satisfying the conclusion of the Baire Category Theorem, and $\M_X$ is a $\s$-ideal for any Baire space $X$. 
The two best known classes of Baire spaces are the completely metrizable spaces and the (locally) compact Hausdorff spaces. Here we focus primarily on the first of these two classes, but we also prove a result about compact Hausdorff spaces in Section~\ref{sec:compacta}.

If $X$ is a Polish space with no isolated points, then $\M_X$ is isomorphic to $\M_\R$: in fact, there is a bijection $f: X \to \R$ such that both $f$ and $f^{-1}$ send meager sets to meager sets. 
It follows that the four cardinal numbers associated to $\M_X$ and $\M_\R$ are the same. 
Henceforth, we use $\M$ (with no subscript) to denote $\M_\R$ or, equivalently (for the purposes of this investigation), $\M_X$ for any perfect Polish space $X$. 

The four cardinal invariants associated to $\M$ are classical cardinal characteristics of the continuum. Their values are between $\aleph_1$ and $\continuum$, and they satisfy the inequalities indicated in the diagram above. 
Furthermore, no inequalities other than these are provable in \zfc: it is consistent that $\non M < \cov M$ (e.g. in the Cohen model) and that $\cov M < \non M$ (e.g. in the random model). 
The ideal $\M$, and its associated cardinal invariants, are fairly well understood when $X$ is Polish.

In this paper we consider the ideal $\M_X$ of meager subsets of a not-necessarily-Polish space $X$. 
In order to avoid trivialities, we always require $X$ to have no isolated points and to be either completely metrizable or compact Hausdorff.
Our motivating question is:

\begin{question*}
Given a space $X$ with no isolated points, can the four cardinals $\mathrm{add}(\mathcal M_X)$, $\mathrm{cov}(\mathcal M_X)$, $\mathrm{non}(\mathcal M_X)$, and $\mathrm{cof}(\mathcal M_X)$ be determined in some simple way from basic topological properties of $X$?
\end{question*}

The majority of the paper is devoted to determining the exact values of three of these four cardinals when $X$ is completely metrizable. The main results along these lines are: if $\k$ is the least cardinal number such that a nonempty open subset of $X$ has weight $\k$, then

\begin{align*}
\mathrm{add}(\mathcal M_X) &\,=\, \begin{cases}
\add M &\text{ if $X$ has a $\pi$-base consisting}\\
&\text{ of second countable open sets,} \\
\aleph_1 &\text{ if not,}
\end{cases} \\
\mathrm{cov}(\mathcal M_X) \vphantom{f^{f^{f^{f^{f^{f^{f^2}}}}}}} &\,=\, \begin{cases}
\cov M &\text{ if } \k = \aleph_0, \\
\aleph_1 &\text{ if } \k > \aleph_0,
\end{cases} \\
\mathrm{non}(\mathcal M_X) \vphantom{f^{f^{f^{f^2}}}} &\,=\, \cf[\k]^\w \cdot \non M.
\end{align*}

\vspace{1mm}

The first two equalities are proved in Section~\ref{sec:covering}. 
The third equality, characterizing $\mathrm{non}(\M_X)$, is proved in Section~\ref{sec:uniformity}. Interestingly, our proof of this last equality relies on a theorem of Shelah from pcf theory. 
We provide an example suggesting that such an approach may be necessary. 

In Section~\ref{sec:cofinality} we prove several inequalities concerning both $\mathrm{cof}(\M_X)$ and the cofinality of the related ideal $\nwd_X$ of nowhere dense subsets of $X$:
$$\k^+,\cf[\k]^\w,\cof M \,\leq\, \mathrm{cof}(\M_X) \,\leq\, \mathrm{cof}(\nwd_X) \,\leq\, 2^\k$$
when $X$ is the completely metrizable space $D^\w$, where $D$ is the discrete space of size $\k$. 
We also include some still-unresolved questions concerning $\mathrm{cof}(\mathcal M_X)$ and $\mathrm{cof}(\nwd_X)$. 

Finally, in Section~\ref{sec:compacta} we discuss (locally) compact Hausdorff spaces. The main result of this section is an upper bound for $\mathrm{non}(\M_X)$  similar to the bound for completely metrizable spaces. Specifically, if $X$ is a locally compact Hausdorff space with no isolated points, and $\k$ is the minimum $\pi$-weight of a nonempty open subset of $X$, then
$$\mathrm{non}(\mathcal M_X) \,\leq\, \cf[\k]^\w \cdot \non M.$$
The proof of this theorem uses our result concerning completely metrizable spaces, and thus also relies indirectly on Shelah's theorem from pcf theory.

\section{Additivity and covering for a completely metrizable space}\label{sec:covering}

Recall that the \emph{weight} of a space $X$, denoted $w(X)$, is the minimum cardinality of a basis for $X$. Let 
$$w_\downarrow(X) \,=\, \min\set{\k}{\text{some nonempty open $U \sub X$ has weight $\k$}}.$$
We say that $X$ has \emph{uniform weight $\k$} if $w(X) = w_\downarrow(X) = \k$, or, equivalently, if every nonempty open subset of $X$ has weight $\k$. 

A \emph{cellular family} in a space $X$ is a collection of nonempty pairwise disjoint open subsets of $X$. The \emph{cellularity} of $X$, denoted $c(X)$, is 
$$c(X) \,=\, \sup \set{ |\mathcal S| }{ \mathcal S \text{ is a cellular family in } X }.$$
By a theorem of Erd\H{o}s and Tarski in \cite{ET}, if $X$ is any space and $c(X)$ is not a regular limit cardinal,
then there is a cellular family in $X$ of cardinality $c(X)$: in other words, the supremum in the definition of $c(X)$ is attained. 
Furthermore, if $X$ is metrizable then the supremum in the definition of $c(X)$ is attained regardless of
whether $c(X)$ is a regular limit cardinal or not. 

In what follows, if a cardinal number $\k$ is referred to as a topological space, assume that $\k$ carries the discrete topology (not its usual order topology). 
So, for example, $\k$ is the discrete space of size $\k$, and $\k^\w$ is a completely metrizable space with uniform weight $\k$. The latter is sometimes called the \emph{Baire space of weight $\k$} (e.g. in \cite{Engelking}). 

\begin{lemma}\label{lem:kw}
Suppose $X$ is a completely metrizable space with uniform weight $\k$. Then $X$ has a co-meager subset homeomorphic to $\k^\w$. 
\end{lemma}

\begin{proof}
Suppose $\k = w(X) = w_\downarrow(X)$. 
By the Bing Metrization theorem, $X$ has a $\s$-discrete basis. 
That is, there is a sequence $\B_0,\B_1,\B_2,\dots$ of countable collections of open subsets of $X$ such that $\B = \bigcup_{n \in \w}\B_n$ is a basis for $X$ and each $\B_n$ is discrete, which means that every point of $X$ has a neighborhood meeting at most one member of $\B_n$. 
Let 
$$Y \,=\, X \setminus \textstyle \bigcup \set{\closure{U} \setminus U}{U \in \B}.$$
In other words, $Y$ is obtained from $X$ by removing the boundary of every member of $\B$. 
We claim that $Y$ is co-meager in $X$ and homeomorphic to $\k^\w$. 

To see that $Y$ is co-meager in $X$, note that $\bigcup \set{\closure{U} \setminus U}{U \in \B_n}$ is closed and nowhere dense in $X$ for each $n \in \w$. (This follows easily from the fact that $\B_n$ is discrete.) Hence $\bigcup \set{\closure{U} \setminus U}{U \in \B} = \bigcup_{n \in \w} \bigcup \set{\closure{U} \setminus U}{U \in \B_n}$ is a meager $F_\s$ set, and $Y$ is a co-meager $G_\dlt$ subset of $X$.

To see that $Y$ is homeomorphic to $\k^\w$, we appeal to a theorem of Stone from \cite{Stone}, which states: 
\emph{A completely metrizable space is homeomorphic to $\k^\w$ if and only if it is strongly $0$-dimensional and has uniform weight $\k$.} 
(This generalizes to higher weights the famous characterization of the Baire space $\w^\w$ by Alexandrov and Urysohn; see \cite[Exercise 7.2.G]{Kechris}.)
The fact that $Y$ is completely metrizable follows from the fact that it is a $G_\dlt$ subset of the completely metrizable space $X$. (Recall that, by a theorem of Alexandrov, a subspace of a completely metrizable space is completely metrizable if and only if it is $G_\dlt$.)
The fact that $Y$ is strongly $0$-dimensional follows from \cite[Theorem 7.3.8]{Engelking}. 
The fact that $Y$ has uniform weight $\k$ follows from the fact that $X$ does. 
Specifically, for every nonempty open $U \sub Y$ there is a nonempty open $W \sub X$ with $U = W \cap Y$. But $U$ is dense in $W$ (because $Y$ is co-meager), and in any $T_3$ space, the weight of a dense subspace is equal to the weight of the original space. Thus $w(U) = w(W) = \k$. 
\end{proof}

Recall that a \emph{$\pi$-base} for $X$ is a collection $\B$ of nonempty open sets such that every nonempty open subset of $X$ contains a member of $\B$. 

\begin{lemma}\label{lem:CellularFamilies}
If $X$ is a metric space and $\B$ is a $\pi$-base for $X$, then there is a maximal cellular family $\mathcal S$ in $X$ such that $\mathcal S \sub \B$ and $|\mathcal S| = w(X)$.
\end{lemma}
\begin{proof}
By the aforementioned result of Erd\H{o}s and Tarski, there is a cellular family $\mathcal S_0$ in $X$ with $|\mathcal S_0| = c(X)$. But $w(X) = c(X)$ for metrizable spaces (see \cite[Theorem 4.1.15]{Engelking}), so $|\mathcal S_0| = w(X)$. 
For every $U \in \mathcal S_0$, a straightforward application of Zorn's Lemma gives a maximal subset $\mathcal T_V \sub \B$ such that $\bigcup \mathcal T_V \sub V$ and the members of $\mathcal T_V$ are pairwise disjoint. 
Because $\B$ is a $\pi$-base, each $\mathcal T_V$ is a maximal cellular family in $V$. But then $\mathcal S = \bigcup_{V \in \mathcal S_0}\mathcal T_V$ is a maximal cellular family in $X$. Furthermore, $c(X) = w(X) = |\mathcal S_0| \leq |\mathcal S| \leq c(X)$, so $|\mathcal S| = w(X)$.
\end{proof}

\begin{lemma}\label{lem:UniformCovering}
Let $X$ be a completely metrizable space with uniform weight $\k \geq \aleph_1$. Then $\mathrm{cov}(\M_X) = \aleph_1$.
\end{lemma}
\begin{proof}
First, we claim that $\mathrm{cov}(\M_{\k^\w}) = \aleph_1$ for uncountable $\k$. To see this, for each $\a < \k$ define 
$$F_\a \,=\, \set{x \in \k^\w}{\a \notin \mathrm{range}(x)}.$$
Here, as usual, we are thinking of $x$ as a function $\w \to \k$. 
Recall that $\k^\w$ has a basis of sets of the form
$$\tr{s} \,=\, \set{x \in \k^\w}{x \rest (\mathrm{length}(s) = s)}$$
where $s \in \k^{<\w}$. 
Each $F_\a$ is closed in $\k^\w$, because
$$\k^\w \setminus F_\a \,=\, \set{x \in \k^\w}{\a \in \mathrm{range}(x)} \,=\, \textstyle \bigcup \set{\tr{s}}{\a \in \mathrm{range}(s)}.$$
Furthermore, if $\tr{s}$ is any element of our basis for $\k^\w$, then $\tr{s \cat \a}$ is another set in the basis with $\tr{s \cat \a} \sub \tr{s}$ and $\tr{s \cat \a} \cap F_\a = \0$. Thus each $F_\a$ is a nowhere dense closed subset of $\k^\w$. 
Now observe that 
$$\textstyle \bigcup_{\a < \w_1} F_\a \,=\, \set{x \in \k^\w}{\mathrm{range}(x) \not\supseteq \w_1} \,=\, X,$$
where the second equality holds because $\mathrm{range}(x)$ is countable for every $x \in \k^\w$. 
Thus we have an $\aleph_1$-sized collection of nowhere dense subsets of $\k^\w$ whose union is all of $\k^\w$.

By Lemma~\ref{lem:kw}, there is a meager $F \sub X$ such that $X \setminus F \homeo \k^\w$. By the previous paragraph, $X \setminus F$ can be covered by $\aleph_1$ sets meager in $X \setminus F$, and therefore meager in $X$. Adding $F$ to this collection, we obtain a collection of $\aleph_1$ meager sets covering $X$.
\end{proof}

\begin{theorem}
If $X$ is a completely metrizable space with no isolated points,
$$\mathrm{cov}(\mathcal M_X) \,=\, \begin{cases}
\cov M &\text{ if $w_\downarrow(X) = \aleph_0$,} \\
\aleph_1 &\text{ if $w_\downarrow(X) > \aleph_0$}.
\end{cases}$$
\end{theorem}
\begin{proof}
For the first case, suppose $w_\downarrow(X) = \aleph_0$. In other words, there is a nonempty open $U \sub X$ with countable weight. Because open subsets of completely metrizable spaces are completely metrizable, $U$ is Polish. Hence $\cov M$ meager-in-$U$ sets are required to cover $U$; equivalently, $\cov M$ nowhere-dense-in-$U$ sets are required to cover $U$. But because $U$ is open, if $F \sub X$ is nowhere dense in $X$ then $U \cap F$ is nowhere dense in $U$. Thus $\geq\!\cov M$ nowhere dense subsets of $X$ are required to cover $U$, and it follows that $\mathrm{cov}(\M_X) \geq \cov M$.

For the opposite inequality, first observe that 
$$\B \,=\, \set{U \sub X}{U \text{ is open, } U \neq \0, \text{ and } w_\downarrow(U) = w(U)}$$
is a $\pi$-base for $X$. By Lemma~\ref{lem:CellularFamilies}, there is a maximal cellular family $\mathcal S$ in $X$ such that $w_\downarrow(U) = w(U)$ for every $U \in \mathcal S$. For each $U \in \mathcal S$, either $(1)$ $U$ is Polish, and there is a sequence $\seq{F^U_\a}{\a < \non M}$ of nowhere dense subsets of $U$ covering $U$, or $(2)$ $U$ is not Polish, $w(U) = w_\downarrow(U) > \aleph_0$, and by Lemma~\ref{lem:UniformCovering} there is a sequence $\seq{F^U_\a}{\a < \w_1}$ of nowhere dense subsets of $U$ covering $U$. 

For each $\a < \w_1$, let $F_a = \bigcup \set{F^U_\a}{U \in \mathcal S}$, and if $\w_1 \leq \a < \cov M$, let $F_\a = \bigcup \set{F^\a_U}{U \in \mathcal S \text{ and } w(U) = \aleph_0}$. Using the fact that $\mathcal S$ is a cellular family, it is not difficult to see that each $F_\a$ is nowhere dense in $X$. 

Because $\mathcal S$ is a maximal cellular family, $\bigcup \mathcal S$ is an open dense subset of $X$; hence $F = X \setminus \bigcup \mathcal S$ is nowhere dense.

Now $\{F\} \cup \set{F_\a}{\a < \mathrm{cov}(\mathcal M)}$ is a collection of $\mathrm{cov}(\mathcal M)$ nowhere dense subsets of $X$, and this collection covers $X$. Hence $\mathrm{cov}(\mathcal M_X) \leq \mathrm{cov}(\M)$. 

For the second case, suppose $w_\downarrow(X) > \aleph_0$. 
Because $X$ is completely metrizable, and therefore satisfies the conclusion of the Baire Category Theorem, $\M_X$ is a $\s$-ideal and $\mathrm{cov}(\M_X) \geq \aleph_1$. 

For the reverse inequality, note that, just as in the first case, there is a maximal cellular family $\mathcal S$ in $X$ such that $w_\downarrow(U) = w(U) > \aleph_0$ for every $U \in \mathcal S$. 
By Lemma~\ref{lem:UniformCovering}, for each $U \in \mathcal S$ there is a sequence $\seq{F^U_\a}{\a < \w_1}$ of nowhere dense subsets of $U$ covering $U$. 
As above, for each $\a < \w_1$ let $F_a = \bigcup \set{F^U_\a}{U \in \mathcal S}$, and let $F = X \setminus \bigcup \mathcal S$.
Then $\{F\} \cup \set{F_\a}{\a < w_1}$ is a collection of $\mathrm{cov}(\mathcal M)$ nowhere dense subsets of $X$, and this collection covers $X$. Hence $\mathrm{cov}(\mathcal M_X) \leq \aleph_1$. 
\end{proof}

\begin{theorem}
If $X$ is a completely metrizable space with no isolated points,
$$\mathrm{add}(\mathcal M_X) \,=\, \begin{cases}
\add M &\text{ if } X \text{ has a $\pi$-base consisting of Polish spaces}, \\
\aleph_1 &\text{ otherwise}.
\end{cases}$$
\end{theorem}
\begin{proof} 
For the first case, suppose $X$ does not have a $\pi$-base consisting of Polish spaces. 
Because 
$$\set{U \sub X}{U \text{ is open, } U \neq \0, \text{ and } w_\downarrow(U) = w(U)}$$
is a $\pi$-base for $X$, there is some nonempty open $U \sub X$ with $w_\downarrow(U) > \aleph_1$. 
By the previous theorem, $\mathrm{cov}(\M_U) = \aleph_1$, i.e., there is a set of $\aleph_1$ meager-in-$U$ sets covering $U$. But every set meager in $U$ is also meager in $X$, so there is a collection of $\aleph_1$ meager subsets of $X$ whose union covers $U$. Because $U$ is non-meager in $X$, this shows $\mathrm{add}(\M_X) \leq \aleph_1$. 
But $\M_X$ is a $\s$-ideal, so $\aleph_1 \leq \mathrm{add}(\M_X)$, and we conclude that $\mathrm{add}(\M_X) = \aleph_1$. 

Suppose next that $X$ does have a $\pi$-base consisting of Polish spaces. 
By Lemma~\ref{lem:CellularFamilies}, there is a maximal cellular family $\mathcal S$ in $X$ consisting of Polish spaces. 
Because each $U \in \mathcal S$ is open, if $Y \sub X$ is meager in $X$ then $Y \cap U$ is meager in $U$. 

Suppose $\F$ is a family of meager subsets of $X$ with $|\F| = \k < \add M$. For each $U \in \mathcal S$, 
$U \cap \bigcup \F = \bigcup \set{U \cap M}{M \in \F}$ is a union of $\k$ meager-in-$U$ sets. Because each $U \in \mathcal S$ Polish and $\k < \add M$, $U \cap \bigcup \F$ is meager in $U$. 
Because $\mathcal S$ is a cellular family, it follows that $\bigcup \set{U \cap \bigcup \F}{U \in \mathcal S} = \big ( \bigcup \mathcal S \big ) \cap \big( \bigcup \mathcal F \big)$ is meager in $X$. But $X \setminus \bigcup \mathcal S$ is also meager in $X$: in fact it is a nowhere dense closed set, because $\bigcup \mathcal S$ is dense in $X$. Thus 
$$\textstyle \bigcup \F \,\sub\, \Big( \big ( \bigcup \mathcal S \big ) \cap \big( \bigcup \mathcal F \big) \Big) \cup \big( X \setminus \bigcup \mathcal S \big)$$
is meager in $X$. Thus $\k < \mathrm{add}(\M_X)$, and this shows $\add M \leq \mathrm{add}(\M_X)$.

For the reverse inequality, fix some nonempty open $U \sub X$ such that $U$ is Polish. 
There is a set $\F$ of meager subsets of $U$ such that $|\F| = \add M$ and $\bigcup \F$ is not meager in $U$.
But because $U$ is open, a subset of $U$ is meager in $U$ if and only if it is also meager in $X$. 
Hence $\F$ is a collection of meager subsets of $X$ such that $|\F| = \add M$ and $\bigcup \F$ is not meager in $X$. 
Thus $\mathrm{add}(\M_X) \leq \add M$.
\end{proof}

\section{Uniformity for a completely metrizable space}\label{sec:uniformity}

In this section we prove the main theorem of the paper, a characterization of $\mathrm{non}(\mathcal M_X)$ in terms of $w_\downarrow(X)$ when $X$ is completely metrizable. 

If $\k$ is an infinite cardinal, then $[\k]^\w$ denotes the set of all countably infinite subsets of $\k$. 
A family $\F \sub [\k]^\w$ is \emph{cofinal} if for every $A \in [\k]^\w$ there is some $B \in \F$ such that $B \supseteq A$. In other words, $\F \sub [\k]^\w$ is cofinal if it is (in the usual sense of the word) cofinal in the poset $\big( [\k]^\w,\sub \!\big)$. Let
$$\cf[\k]^\w \,=\, \min\set{|\F|}{\F \text{ is a cofinal subset of }[\k]^\w}.$$
For example, $\cf[\w_1]^\w = \aleph_1$ because, identifying an ordinal with the set of its predecessors as usual, $\set{\a}{\w \leq \a < \w_1}$ is a cofinal subset of $[\w_1]^\w$. 
In fact, this example with $\w_1$ can be extended to show that $\cf[\k]^\w = \k$ for all uncountable $\k < \aleph_\w$. 

As for cardinals $>\!\aleph_\w$, it is not difficult to see that $\k \leq \cf[\k]^\w$ for all uncountable $\k$, and a reasonably straightforward diagonalization argument shows also that $\k^+ \leq \cf[\k]^\w$ whenever $\k$ is a singular cardinal with cofinality $\w$. 
\emph{Shelah's Strong Hypothesis}, abbreviated $\mathsf{SSH}$, is the statement that these straightforward lower bounds are sharp:
$$\text{\scalebox{.95}{$(\mathsf{SSH})\!:$}} \qquad \qquad \qquad \cf[\k]^\w \,=\, \begin{cases}
\k &\text{ if } \cf(\k) > \w,\\
\k^+ &\text{ if } \cf(\k) = \w.
\end{cases} \qquad \qquad \qquad \qquad \qquad$$
The failure of $\mathsf{SSH}$ implies the consistency of large cardinal axioms (see \cite{JMPS} or \cite{Brian}). Nonetheless, it is consistent, relative to large cardinals, that $\cf[\k]^\w$ takes on values larger than those listed above (see \cite{Matet}). 

Following \cite{Jech1}, a subset $\mathcal C$ of $[\k]^\w$ is \emph{closed} if whenever $\set{A_n}{n \in \w} \sub \C$ and $A_0 \sub A_1 \sub A_2 \sub \dots$, then $\bigcup_{n \in \w}A_n \in \C$. If $\C \sub [\k]^\w$ is both closed and cofinal, then $\C$ is called a \emph{club}. (In \cite{Jech1} the word ``unbounded'' is used rather than ``cofinal''.) A subset of $[\k]^\w$ is \emph{stationary} if it intersects every club subset of $[\k]^\w$. 

Given a sequence of functions $\vec g = \seq{g_n}{n \in \w}$ mapping $\k^{<\w} \to \k^{<\w}$, let
$$\C(\vec g\,) \,=\, \set{A \in [\k]^\w}{\text{ for every $n \in \w$, if } s \in A^{<\w} \text{ then } g_n(s) \in A^{<\w}}.$$
In other words, $\C(\vec g\,)$ is the set of all $A \in [\k]^{\w}$ that are closed, in the appropriate sense, with respect to all the $g_n$. 
The following result can be found in \cite{Jech2} or \cite{Shioya}, and is proved by a standard closing-off argument: 

\begin{lemma}\label{lem:Jech}
For any sequence $\vec g = \seq{g_n}{n \in \w}$ of functions $\k^{<\w} \to \k^{<\w}$, $\C(\vec g\,)$ is club in $[\k]^\w$. In particular, every stationary subset of $[\k]^\w$ contains a member of $\C(\vec g\,)$.
\end{lemma}

The main tool for proving our characterization of $\mathrm{non}(\M_X)$ in this section is the following result of Shelah.

\begin{theorem}[Shelah]\label{thm:Shelah}
For any infinite cardinal $\k$, there is a stationary subset of $[\k]^\w$ with cardinality $\cf[\k]^\w$.
\end{theorem}

All stationary subsets of $[\k]^\w$ are cofinal: this follows from the fact that for any given $A \in [\k]^\w$, $\set{B \in [\k]^\w}{B \supseteq A}$ is club, and so any given stationary set must contain some $B \supseteq A$. 
Consequently, every stationary set has size $\geq\! \cf[\k]^\w$. Shelah's theorem says this lower bound is sharp. 

All known proofs of Theorem~\ref{thm:Shelah} rely on pcf theory. 
A proof first appeared in \cite{Shelah1}, and the result also appears in Shelah's book \cite{Shelah2}. A simplified and corrected version of the proof is presented in \cite{Shioya}. 

It is worth pointing out that if $\k = \aleph_n < \aleph_\w$ then Theorem~\ref{thm:Shelah} can be proved for $\k$ without recourse to pcf theory, via a relatively straightforward induction on $n$. Indeed, this special case of the theorem was proved by Baumgartner and Taylor in \cite{BT} before the invention of pcf theory. 

In Remark~\ref{rem:pcf} below, we provide an example that suggests the use of stationary sets, and Shelah's Theorem~\ref{thm:Shelah}, may be necessary for proving our characterization of $\mathrm{non}(\M_X)$.

\begin{theorem}\label{thm:UW}
Let $X$ be a completely metrizable space with uniform weight $\k$. Then $\mathrm{non}(\M_X) = \non M \cdot \cf[\k]^\w$. 
\end{theorem}
\begin{proof}
If $X$ has uniform weight $\k$, then by Lemma~\ref{lem:kw} $X$ has a co-meager subspace homeomorphic to $\k^\w$. But if $Y$ is co-meager in $X$, then $Z$ is meager in $X$ if and only if $Z \cap Y$ is meager in $Y$. Thus, to prove the theorem it suffices to show that $\mathrm{non}(\M_{\k^\w}) = \non M \cdot \cf[\k]^\w$.

To see that $\non M \leq \mathrm{non}(\M_{\k^\w})$, note that $\k^\w \homeo (\k \times \w)^\w \homeo \k^\w \times \w^\w$. Let $\pi$ denote the natural projection $\k^\w \times \w^\w \to \w^\w$. If $Z \sub \k^\w \times \w^\w$ and $|Z| < \non M$, then $|\pi[Z]| < \non M$ and so $\pi[Z]$ is meager in $\w^\w$. Thus there is a sequence $F_0,F_1,F_2,\dots$ of nowhere dense subsets of $\w^\w$ with $\pi[Z] \sub \bigcup_{n \in \w}F_n$. But it is clear that if $F$ is nowhere dense in $\w^\w$ then $\pi^{-1}(F)$ is nowhere dense in $\k^\w$. 
Therefore $\pi^{-1}(F_0),\pi^{-1}(F_1),\pi^{-1}(F_2),\dots$ is a sequence of nowhere dense subsets of $\k^\w \times \w^\w$ with $Z \sub \pi^{-1}(\pi[Z]) \sub \bigcup_{n \in \w}\pi^{-1}(F_n)$. Thus $Z$ is meager. Hence 
$\non M \leq \mathrm{non}(\M_{\k^\w \times \w^\w}) = \mathrm{non}(\M_{\k^\w})$. 

To see that $\cf[\k]^\w \leq \mathrm{non}(\M_{\k^\w})$, fix some $Z \sub \k^\w$ with $|Z| < \cf[\k]^\w$. 
Viewing the points of $\k^\w$ as functions $\w \to \k$, we have $|\set{\mathrm{range}(z)}{z \in Z}| \leq |Z| < \cf[\k]^\w$, and therefore $\set{\mathrm{range}(z)}{z \in Z}$ is not cofinal in $[\k]^\w$. 
Thus there is some $A \in [\k]^\w$ such that $\mathrm{range}(z) \not\supseteq A$ for any $z \in Z$. 

As in the proof of Lemma~\ref{lem:UniformCovering}, let $F_\a = \set{x \in \k^\w}{a \notin \mathrm{range}(x)}$ for each $\a < \k$. 
In the proof of Lemma~\ref{lem:UniformCovering}, we showed that each of these $F_\a$ is closed and nowhere dense in $\k^\w$. Consequently
$$\textstyle \bigcup_{\a \in A} F_\a \,=\, \set{x \in \k^\w}{\mathrm{range}(x) \not\supseteq A}$$
is meager in $\k^\w$. But this set contains $Z$, by the previous paragraph. 
Thus $Z$ is meager, and it follows that 
$\cf[\k]^\w \leq \mathrm{non}(\M_{\k^\w})$. 

It remains to prove the reverse inequality: $\mathrm{non}(\M_{\k^\w}) \leq \non M \cdot \cf[\k]^\w$. To do this, we find a non-meager subset of $\k^\w$ with size $\non M \cdot \cf[\k]^\w$. 

For each $A \in [\k]^\w$, observe that $A^\w$ is a closed subset of $\k^\w$, and is homeomorphic to the Baire space $\w^\w$. In particular, if $A \in [\k]^\w$ then there is some $Z_A \sub A^\w$ such that $Z_A$ is non-meager in $A^\w$ and $|Z_A| = \non M$.  
Applying Theorem~\ref{thm:Shelah}, let $\mathcal S$ be a stationary subset of $[\k]^\w$ with $|\mathcal S| = \cf[\k]^\w$. 
Let
$Z = \bigcup_{A \in \mathcal S}Z_A$. Clearly $|Z| = |\mathcal S| \cdot \non M = \non M \cdot \cf[\k]^\w$, and we claim that $Z$ is a non-meager subset of $\k^\w$.

Aiming for a contradiction, suppose instead that $Z$ is meager in $\k^\w$. 
Fix a sequence $F_0,F_1,F_2,\dots$ of closed nowhere dense subsets of $\k^\w$ with $Z \sub \bigcup_{n \in \w}F_n$. 
Recall that $\k^\w$ has a basis of sets of the form
$$\tr{s} \,=\, \set{x \in \k^\w}{x \rest (\mathrm{length}(s) = s)}$$
where $s \in \k^{<\w}$. 
Fix $n \in \w$. 
For each basic (cl)open set $\tr{s}$, there is a basic (cl)open set $\tr{t} \sub \tr{s}$ such that $\tr{t} \cap F_n = \0$, because $F_n$ is nowhere dense.  
For each $n \in \w$, fix a function $g_n: \k^{<\w} \to \k^{<\w}$ choosing some such $t$ for each $s$: i.e., for each $s \in \k^{<\w}$, $\tr{g_n(s)} \sub \tr{s}$ and $\tr{g_n(s)} \cap F_n = \0$. 

By Lemma~\ref{lem:Jech}, 
$$\C \,=\, \set{A \in [\k]^\w}{\text{for any $n \in \w$, if } x \in A^{<\w} \text{ then } g_n(s) \in A^{<\w} }$$ 
is a club subset of $[\k]^\w$. Using the fact that $\mathcal S$ is stationary, fix $A \in \mathcal S \cap \C$. 

For any $s \in \k^{<\w}$, obesrve that $\tr{s} \cap A^\w \neq \0$ if and only if $s \in A^{<\w}$. So the basic (cl)open neighborhoods of $A^\w$ are the sets of the form $\tr{s} \cap A$ with $s \in A^{<\w}$. 
Fix $n \in \w$. 
If $s \in A^{<\w}$, then $g_n(s) \in A^{<\w}$, and therefore $\tr{g_n(s)} \cap A^\w$ is a nonempty (cl)open subset of $A^\w$ such that $\tr{g_n(s)} \sub \tr{s}$ and $\tr{g_n(s)} \cap F_n = \0$. As $s$ was arbitrary, this shows that $F_n \cap A^\w$ is nowhere dense in $A^\w$. Because this is true for all $n$ and $Z \sub \bigcup_{n \in \w}F_n$, this shows that $Z \cap A^\w$ is meager in $A^\w$. 
This is a contradiction, because $Z \supseteq Z_A$ and $Z_A$ is non-meager in $A^\w$. 

Hence $Z$ is a non-meager subset of $\k^\w$ with size $\non M \cdot \cf[\k]^\w$, and therefore $\mathrm{non}(\M_{\k^\w}) \leq \non M \cdot \cf[\k]^\w$.
\end{proof}

\begin{theorem}\label{thm:MetUnif}
Let $X$ be a completely metrizable space, and let $\k = w_\downarrow(X)$. Then $\mathrm{non}(\M_X) = \non M \cdot \cf[\k]^\w$.
\end{theorem}
\begin{proof}
Fix some nonempty open $U \sub X$ with $w(U) = w_\downarrow(U) = \k$. By Theorem~\ref{thm:UW}, there is some $Z \sub U$ such that $Z$ is non-meager in $U$ and $|Z| = \non M \cdot \cf[\k]^\w$. But if $Z$ is non-meager in $U$ and $U$ is open in $X$, then $Z$ is non-meager in $X$. Thus $\mathrm{non}(\M_X) \leq \non M \cdot \cf[\k]^\w$. 

For the reverse inequality, fix some $Z \sub X$ with $|Z| < \non M \cdot \cf[\k]^\w$. By Lemma~\ref{lem:CellularFamilies}, there is a maximal cellular family $\mathcal S$ in $X$ such that $w_\downarrow(U) = w(U)$ for every $U \in \mathcal S$. 
For each $U \in \mathcal S$, we have $\lambda = w(U) \geq \k$, which implies $\cf[\lambda]^\w \geq \cf[\k]^\w$. 
Thus $|Z \cap U| < \non M \cdot \cf[\k]^\w \leq \non M \cdot [\lambda]^\w$, and it follows from Theorem~\ref{thm:UW} that $Z \cap U$ is meager in $U$. 

For each $U \in \mathcal S$, fix a sequence $F_1^U,F_2^U,F_3^U,\dots$ of nowhere dense sets with $\bigcup_{n \in \w}F_n \supseteq Z \cap U$. Because $\mathcal S$ is a cellular family, $F_n = \bigcup_{U \in \mathcal S}F_n^U$ is nowhere dense in $X$ for each $n$. Also, $F_0 = X \setminus \bigcup \mathcal S$ is nowhere dense, because $\mathcal S$ is a maximal cellular family. Thus $F_0,F_1,F_2,\dots$ is a sequence of nowhere dense subsets of $X$, and $Z \sub \bigcup_{n \in \w}F_n$. Hence $Z$ is meager, and $\mathrm{non}(\M_X) \geq \non M \cdot \cf[\k]^\w$.
\end{proof}

\begin{remark}\label{rem:pcf}
Shelah's Theorem~\ref{thm:Shelah} enters this proof as a tool for showing the inequality 
$$\mathrm{non}(\M_{\k^\w}) \,\leq\, \non M \cdot \cf[\k]^\w,$$
i.e., for showing that $\k^\w$ has a non-meager subset of size $\non M \cdot \cf[\k]^\w$. 
In summary, our idea was: let $\F$ be a stationary family in $[\k]^\w$, and for each $A \in \F$ let $Z_A$ be a non-meager subset of $A^\w$ with size $\non M$. (Note that $A^\w$ is just a copy of the Baire space, so it does in fact have a non-meager subset of this size.) Then $\bigcup_{A \in \F}Z_A$ is non-meager. 

This raises the question of whether we really need $\F$ to be stationary. Would any cofinal family have sufficed? 
The purpose of this remark is to give an example answering this question in the negative: if $\F$ is merely cofinal in $[\k]^\w$, then $\bigcup_{A \in \F}Z_A$ may not be non-meager in $\k^\w$. 

To see this, fix a cardinal $\k > \w$ and let
$\F \,=\, \set{A \in [\k]^\w}{\sup A \in A}.$ 
It is clear that $\F$ is cofinal in $[\k]^\w$. However, a little thought reveals that 
$\C = \set{A \in [\k]^\w}{\sup A \notin A}$ is club, and thus $\F$ is non-stationary. 

For each $A \in \F$, let 
$$U_A \,=\, \set{x \in A^\w}{\sup A \in \mathrm{range}(x)} \,=\, \textstyle \bigcup \set{\tr{s}}{\sup A \in \mathrm{range}(s)}.$$
This set is open in $A^\w$. Furthermore, if $\tr{s}$ is a basic (cl)open subset of $A^\w$, then $\tr{s \cat (\sup A)}$ is a basic (cl)open subset of $A^\w$ contained in $U_A \cap \tr{s}$. Thus $U_A$ is a dense open subset of $U_A$. Because $A^\w$ is just a copy of the Baire space, there is some $Z_A \sub U_A$ with $|Z_A| = \non M$. 

We claim that $\bigcup_{A \in \F}Z_A$ is meager in $\k^\w$. To see this, let
$$F_n \,=\, \set{x \in \k^\w}{x(n) = \sup \mathrm{range}(x)}.$$ 
Each $F_n$ is a closed subset of $\k^\w$, because
$$\k^\w \setminus F_n \,=\, \textstyle \bigcup \set{\tr{s}}{\mathrm{length}(s) > n \text{ and } s(n) \neq \sup \mathrm{range}(s)}.$$
Furthermore, if $\tr{s}$ is a basic (cl)open subset of $\k^\w$, then we can find some $s' \supseteq s$ with $\mathrm{length}(s') > n$, and then set $t = s' \cat \a$ for some $\a > \sup \mathrm{range}(s')$. Then $\tr{t}$ is a basic (cl)open subset of $\k^\w$ with $\tr{t} \sub \tr{s}$ and $\tr{t} \cap F_n = \0$. Thus each $F_n$ is nowhere dense in $\k^\w$. Consequently $\bigcup_{n \in \w}F_n$ is meager. Furthermore, $U_A \sub \bigcup_{n \in \w}F_n$ for each $A \in \F$, and this means that $Z = \bigcup_{A \in \F}Z_A \sub \bigcup_{A \in \F}U_A \sub \bigcup_{n \in \w}F_n$. Therefore $Z$ is meager. 
\hfill{\tiny $\square$}
\end{remark}

\section{The cofinality of $\M_X$ and $\nwd_X$}\label{sec:cofinality}

Given an ideal $\I$ on a set $X$, we say $\F \sub \I$ is \emph{cofinal} if  $\F$ is cofinal (in the usual sense of the word) in the poset $(\I,\sub)$; i.e., for every $A \in \I$ there is some $B \in \F$ such that $B \supseteq A$. 
Recall from the introduction that $\cof \I$ is the minimal size of a cofinal subset of $\I$. 

Given a topological space $X$, let
$$\nwd_X \,=\, \set{N \sub X}{N \text{ is nowhere dense in } X}.$$
If $X$ is a $T_1$ space with no isolated points then $\nwd_X$ is an ideal on $X$. It is not generally a $\s$-ideal. 

It is not difficult to see that $\mathrm{cov}(\nwd_X) = \mathrm{cov}(\mathcal M_X)$ whenever $\M_X$ is a $\s$-ideal, and therefore, by the results in the previous section,
$$\mathrm{cov}(\nwd_X) \,=\, \mathrm{cov}(\mathcal M_X) \,=\, \begin{cases}
\cov M &\text{ if } w_\downarrow(X) = \aleph_0 \\
\aleph_1 &\text{ if } w_\downarrow(X) > \aleph_0
\end{cases}$$
if $X$ is completely metrizable and has no isolated points. 
Using the fact that the density of a metric space is equal to its weight (see \cite[Theorem 4.1.15]{Engelking}), 
\begin{align*}
\mathrm{non}(\nwd_X) &\,=\, \min \set{|D|}{D \text{ is dense in a nonempty open $U \sub X$}} \\ 
&\,=\, \set{w(U)}{U \sub X \text{ is nonempty and open}} \,=\, w_\downarrow(X)
\end{align*}
for every metrizable space $X$ (whether it has isolated points or not). 
If $X$ is metrizable and has no isolated points, then taking $F_n$ to be a maximal $\nicefrac{1}{n}$-separated subset of $X$ for each $n$, we obtain a countable collection of nowhere dense sets whose union is dense. Thus 
$$\mathrm{add}(\nwd_X) \,=\, \aleph_0.$$
Hence, as with $\M_X$, three of the four cardinal invariants of $\nwd_X$ are determined, all but the cofinality.

In this section we prove some results and ask some questions concerning $\mathrm{cof}(\nwd_{\k^\w})$ and $\mathrm{cof}(\M_{\k^\w})$. The reason for focusing on the generalized Baire spaces $\k^\w$ is twofold: (1) these spaces played a fundamental part in our analysis of additivity, covering, and uniformity in the preceding two sections, and (2) already the spaces $\k^\w$ present significant difficulties and raise some interesting questions. 
The foremost of these questions is: 

\begin{question}\label{q:nwdvsM}
Is $\mathrm{cof}(\nwd_{\k^\w}) = \mathrm{cof}(\M_{\k^\w})$ for every $\k$?
\end{question}

For the case $\k = \w$ (the Baire space) the answer to this question is yes, by a theorem of Fremlin in \cite[Section 3]{Fremlin}. For uncountable $\k$ the question remains open (even for $\k = \w_1$). However, we do have the following inequality:

\begin{theorem}\label{thm:≤}
$\mathrm{cof}(\M_{\k^\w}) \leq \mathrm{cof}(\nwd_{\k^\w})$ for every $\k$.
\end{theorem}
\begin{proof}
For each $n \in \w$, let 
$B_n = \set{x \in \k^\w}{x(0) = n}$, and let
$\phi_n: B_n \to \k^\w$ be the natural homeomorphism sending $x \in B_n$ to the function $n \mapsto x(n+1)$ (in other words, $\phi_n$ deletes the first coordinate of $x$). 
Given $F \sub \k^\w$, define $\tilde F = \bigcup_{n \in \w} \phi_n[F \cap B_n]$. 
It is clear that the map $F \mapsto \tilde F$ sends nowhere dense subsets of $\k^\w$ to meager subsets of $\k^\w$. We claim that if $\F$ is a cofinal subset of $\nwd_{\k^\w}$, then $\{\tilde F :\, F \in \F\}$ is cofinal in $\M_{\k^\w}$. 

To see this, let $Z$ be a meager subset of $\k^\w$, and let $F_0,F_1,F_2,\dots$ be a sequence of nowhere dense sets covering $Z$. But then $F = \bigcup_{n \in \w} \phi_n^{-1}(F_n)$ is nowhere dense in $\k^\w$, and $\tilde F = \bigcup_{n \in \w}F_n \supseteq Z$. Because $\F$ is cofinal, there is some $E \in \F$ with $E \supseteq F$, and then $\tilde E \supseteq Z$.
Thus $\{\tilde F :\, F \in \F\}$ is cofinal in $\M_{\k^\w}$, and this shows $\mathrm{cof}(\M_{\k^\w}) \leq \mathrm{cof}(\nwd_{\k^\w})$.
\end{proof}

We next proceed to prove three lower bounds for $\mathrm{cof}(\M_{\k^\w})$ and one upper bound for $\mathrm{cof}(\nwd_{\k^\w})$: 
$$\k^+,\cf[\k]^\w,\cof M \,\leq\, \mathrm{cof}(\M_{\k^\w}) \,\leq\, \mathrm{cof}(\nwd_{\k^\w}) \,\leq\, 2^\k.$$
Two of these lower bounds ($\k^+$ and $\cf[\k]^\w$) are straightforward. The third requires the following lemma reminiscent of the Kuratowski-Ulam theorem.

\begin{lemma}\label{lem:Game}
Suppose $X$ and $Y$ are completely metrizable spaces, and let $\pi_Y$ denote the natural projection $X \times Y \to Y$. Then $A \sub Y$ is meager in $Y$ if and only if $\pi_Y^{-1}(A)$ is meager in $X \times Y$.
\end{lemma}
\begin{proof}
For the ``only if'' direction of the lemma, suppose $A$ is a meager subset of $Y$. Let $F_0,F_1,F_2,\dots$ be a sequence of nowhere dense subsets of $Y$ with $A \sub \bigcup_{n \in \w}F_n$. But then $\pi_Y^{-1}(F_0), \pi_Y^{-1}(F_1), \pi_Y^{-1}(F_2),\dots$ is a sequence of nowhere dense subsets of $X \times Y$ and $\pi_Y^{-1}(A) \sub \bigcup_{n \in \w}\pi_Y^{-1}(F_n)$. Thus, if $A$ is meager then $\pi_Y^{-1}(A)$ is meager too.

For the ``if'' direction, we make use of the Banach-Mazur games $\mathsf{BM}(Y,A)$ and $\mathsf{BM}(X \times Y,\pi_Y^{-1}(A))$. 
Recall that the game $\mathsf{BM}(Y,A)$ is played in $\w$ rounds, where two players, \HIT and {\small M}{\scriptsize ISS}, take turns choosing members of an infinite decreasing sequence of open subsets of $Y$ as follows:

\begin{center}
\begin{tikzpicture}[xscale=.95]

\node at (-1.6,1) {\scalebox{.85}{{\small H}{\scriptsize IT}}};
\node at (-1.671,0) {\scalebox{.85}{{\small M}{\scriptsize ISS}}};
\node at (0,1) {\small $U_0$};
\node at (1,0) {\small $V_0$};
\node at (2,1) {\small $U_1$};
\node at (3,0) {\small $V_1$};
\node at (4,1) {\small $U_2$};
\node at (5,.5) {$\dots$};
\node at (6,0) {\small $V_{n-1}$};
\node at (7,1) {\small $U_n$};
\node at (8,0) {\small $V_n$};
\node at (9,1) {\small $U_{n+1}$};
\node at (10,.5) {$\dots$};
\node at (.5,.5) {\footnotesize \rotatebox{315}{$\supseteq$}};
\node at (1.5,.5) {\footnotesize \rotatebox{45}{$\supseteq$}};
\node at (2.5,.5) {\footnotesize \rotatebox{315}{$\supseteq$}};
\node at (3.5,.5) {\footnotesize \rotatebox{45}{$\supseteq$}};
\node at (6.5,.5) {\footnotesize \rotatebox{45}{$\supseteq$}};
\node at (7.5,.5) {\footnotesize \rotatebox{315}{$\supseteq$}};
\node at (8.5,.5) {\footnotesize \rotatebox{45}{$\supseteq$}};

\end{tikzpicture}
\end{center}

\noindent In the end, \MISS wins if and only if $A \cap \bigcap_{n=1}^\infty U_n = A \cap \bigcap_{n=1}^\infty V_n = \0$. 
A \emph{strategy} for \MISS is a function $\t$ whose domain consists of possible finite partial plays of the game, i.e. decreasing sequences $\< U_0,V_0,\dots,U_{n-1},V_{n-1},U_n\>$ of nonempty open subsets of $Y$, such that $\t(\< U_0,V_0,\dots,U_{n-1},V_{n-1},U_n\>)$ is a nonempty open subset of $U_n$. 
We say $\t$ is a \emph{winning strategy} for \MISS if playing $V_n = \t(\< U_0,V_0,\dots,U_{n-1},V_{n-1},U_n\>)$ in every round of $\mathsf{BM}(Y,A)$ results in \MISS winning the game. Of course, all these terms are defined similarly for $\mathsf{BM}(X \times Y,\pi_Y^{-1}(A))$ as well.

Recall that \MISS has a winning strategy in the game $\mathsf{BM}(Y,A)$ if and only if $A$ is meager in $Y$, and similarly, \MISS has a winning strategy in $\mathsf{BM}(X \times Y,\pi_Y^{-1}(A))$ if and only if $\pi_Y^{-1}(A)$ is meager in $X \times Y$. 
Thus, to finish the proof of the lemma, it suffices to show that if \MISS has a winning strategy in $\mathsf{BM}(X \times Y,\pi_Y^{-1}(A))$, then \MISS has a winning strategy in $\mathsf{BM}(Y,A)$.

To do this, suppose \MISS has a winning strategy $\t$ in $\mathsf{BM}(X \times Y,\pi_Y^{-1}(A))$. 
We will describe how to use this winning strategy in $\mathsf{BM}(X \times Y,\pi_Y^{-1}(A))$ to create a winning strategy for \MISS in $\mathsf{BM}(Y,A)$. 
For the sake of clarity we do not explicitly write the strategy in $\mathsf{BM}(Y,A)$ as a function: we simply describe how \MISS should play in $\mathsf{BM}(Y,A)$. 

In round $0$ of the game $\mathsf{BM}(Y,A)$, \HIT plays some $U_0 \sub Y$. \MISS pretends that \HIT has just played $\tilde U_0 = \pi_Y^{-1}(U_0)$ in the game $\mathsf{BM}(X \times Y,\pi_Y^{-1}(A))$, and responds with some $\t(\tilde U_0) = \tilde V_0 \sub \tilde U_0$, according to {\small M}{\scriptsize ISS}'s winning strategy for that game. Next, we find open sets $O_X^0 \sub X$ and $O_Y^0 \sub Y$ such that 
\begin{enumerate}
\item the diameter of $O_X^0$ is $<\! 1$, and
\item $O_X^0 \times O_Y^0 \sub \tilde V_0$. 
\end{enumerate}
Finally, we return to the game $\mathsf{BM}(Y,A)$ and \MISS plays $V_0 = O_Y^0$. This is a legal play for \MISS in $\mathsf{BM}(Y,A)$, because $O_X^0 \times O_Y^0 \sub \tilde V_0 \sub \tilde U_0 = \pi_Y^{-1}(U_0)$ and this implies $O_Y^0 \sub U_0$. 

(Observe that, if we really want a strategy for {\small M}{\scriptsize ISS}, we need a function mapping $U_0$ to $V_0$, and the vague phrase ``find open subsets $O_X^0 \sub X$ and $O_Y^0 \sub Y$ . . .'' doesn't really define a function. However, we can get around this simply by fixing beforehand some choice function $F$ that can always ``find'' such a pair in any such situation, and then the new strategy is a well-defined function, whose definition uses $F$ as a parameter.)

Subsequent rounds of the game proceed similarly. In round $n$ of the game, \MISS has already played some rectangle $O_X^{n-1} \times O^{n-1}_Y$ in round $n-1$ of the auxiliary game $\mathsf{BM}(X \times Y,\pi^{-1}(A))$, and then played $V_{n-1} = O^{n-1}_Y$ in round $n-1$ of $\mathsf{BM}(Y,A)$. 
Now suppose \HIT plays some $U_n \sub V_{n-1}$ to begin round $n$ of $\mathsf{BM}(Y,A)$. 
Similarly to what happened in round $0$, \MISS pretends that \HIT has just played $\tilde U_n = \pi_Y^{-1}(U_n) \cap (O_X^{n-1} \times O^{n-1}_Y)$ in round $n$ of the auxiliary game $\mathsf{BM}(X \times Y,\pi_Y^{-1}(A))$. Using the winning strategy $\t$ for that game, \MISS determines a response $\t(\tilde U_n) = \tilde V_n \sub \tilde U_n$. We then find open sets $O_X^n \sub X$ and $O_Y^n \sub Y$ such that 
\begin{enumerate}
\item the diameter of $O_X^n$ is $<\!\nicefrac{1}{n}$,
\item $\closure{O_X^n} \sub O_X^{n-1}$, and
\item ${O_X^n \times O_Y^n} \sub \tilde V_n$. 
\end{enumerate}
Finally, we return to the game $\mathsf{BM}(Y,A)$ and \MISS plays $V_n = O_Y^n$. 

We claim that this is a winning strategy for \MISS in $\mathsf{BM}(Y,A)$. 
To see this, first observe that $\< \tilde U_0,\tilde V_0,\tilde U_1,\tilde V_1, \tilde U_2,\tilde V_2,\dots\>$ is a legal play of the game $\mathsf{BM}(X \times Y,\pi^{-1}(A)_Y)$ in which \MISS uses the winning strategy $\t$. Thus 
$$\textstyle \bigcap_{n=1}^\infty \tilde U_n \,=\, \bigcap_{n=1}^\infty \tilde V_n \,=\, \bigcap_{n=1}^\infty (O^n_X \times O^n_Y)$$ 
does not contain any points of $\pi^{-1}_Y(A)$. 

Requirements $(1)$ and $(2)$ on the sets $O^n_X$, together with the fact that $X$ is completely metrizable, ensure that $\bigcap_{n=1}^\infty O^n_X \neq \0$, and in fact, this set contains a single point, say $x$. 
But then $y \in \bigcap_{n=1}^\infty V_n = \bigcap_{n=1}^\infty O^n_Y$ if and only if $(x,y) \in \bigcap_{n=1}^\infty O^n_X \times O^n_Y$. 
By the previous paragraph, this means that if $(x,y) \in \bigcap_{n=1}^\infty O^n_X \times O^n_Y$ then $(x,y) \notin \pi^{-1}_Y(A)$, i.e., $y \notin A$. 
Thus if $y \in \bigcap_{n=1}^\infty V_n$ then $y \notin A$. In other words, playing in this way has resulted in a win for \MISS in $\mathsf{BM}(Y,A)$. 
\end{proof}

\begin{theorem}
For any infinite cardinal $\k$,
$$\k^+, \cf[\k]^\w, \cof \M \,\leq\, \mathrm{cof}(\M_{\k^\w}) \,\leq\, \mathrm{cof}(\nwd_{\k^\w}) \,\leq\, 2^\k.$$
\end{theorem}
\begin{proof}
To see that $\k^+ \leq \mathrm{cof}(\M_{\k^\w})$, let $\F$ be a family of meager subsets of $\k^\w$ with $|\F| = \k$. Enumerate $\F = \set{F_\a}{\a < \k}$, and for each $\a < \k$ let 
$B_\a = \set{x \in \k^\w}{x(0) = \a}$. 
Because each $B_\a$ is clopen in $\k^\w$ and $F_\a$ is nowhere dense, for each $\a$ there is some $x_\a \in B_\a \setminus F_\a$. But then $\set{x_\a}{\a < \k}$ is meager (in fact it is a closed discrete set) and is not contained in any member of $\F$. Thus $\F$ is not cofinal. 

To see that $\cf[\k]^\w \leq \mathrm{cof}(\M_{\k^\w})$, simply note that $\mathrm{non}(\M_{\k^\w}) \leq \mathrm{cof}(\M_{\k^\w})$, and $\cf[\k]^\w \leq \mathrm{non}(\M_{\k^\w})$ by Theorem~\ref{thm:UW}.

To see that $\cof M \leq \mathrm{cof}(\M_{\k^\w})$, first note, as in the proof of Theorem~\ref{thm:UW}, that $\k^\w \homeo (\k \times \w)^\w \homeo \k^\w \times \w^\w$. Let $\pi: \k^\w \times \w^\w \to \w^\w$ denote the natural projection. 
Given $A \sub \k^\w \times \w^\w$, define 
$\tilde A = \set{x \in \w^\w}{\pi^{-1}(x) \sub A}.$ 

If $A \sub \k^\w \times \w^\w$, then $A \supseteq \pi^{-1}(\tilde A)$. Using this fact, Lemma~\ref{lem:Game} implies that if $A$ is meager in $\k^\w \times \w^\w$ then $\tilde A$ is meager in $\w^\w$. 

Now suppose $\F$ is a cofinal subset of $\M_{\k^\w \times \w^\w}$. By the previous paragraph, $\{ \tilde A :\, A \in \F \}$ is a subset of $\M_{\w^\w}$. We claim that it is also cofinal. 
To see this, let $X \sub \w^\w$ be meager. 
By the easy direction of Lemma~\ref{lem:Game}, $\pi^{-1}(X)$ is meager in $\k^\w \times \w^\w$. 
Consequently, there is some $A \in \F$ with $A \supseteq \pi^{-1}(X)$, and this implies $\tilde A \supseteq X$. 
Thus $\{ \tilde A :\, A \in \F \}$ is a cofinal subset of $\M_{\w^\w}$, and it follows that $\cof M \leq \mathrm{cof}(\M_{\k^\w \times \w^\w}) = \mathrm{cof}(\M_{\k^\w})$. 

Finally, we wish to show that $\mathrm{cof}(\nwd_{\k^\w}) \leq 2^\k$. This is straightforward: $2^\k$ has a basis of size $\k$ (the sets of the form $\tr{s}$) and thus has $2^\k$ open sets and $2^\k$ closed sets. The inequality follows, because the collection of nowhere dense closed sets is cofinal in $\nwd_X$.
\end{proof}

Let us note that these three lower bounds are independent of one another, in the sense that any of $\k^+$, $\cf[\k]^\w$, or $\cof M$ can be strictly larger than the other two, for certain values of $\k$ in certain models of set theory. 

If \gch holds, this theorem implies $\k^+ = 2^\k = \mathrm{cof}(\M_{\k^\w}) = \mathrm{cof}(\nwd_{\k^\w})$ for every $\k$. This gives a consistent positive answer to Question~\ref{q:nwdvsM}. 
Combining this observation with Theorem~\ref{thm:≤}, the open part of Question~\ref{q:nwdvsM} is whether it is consistent to have $\mathrm{cof}(\M_{\k^\w}) < \mathrm{cof}(\nwd_{\k^\w})$ for some $\k$. 

\begin{question}
Is it consistent that $\max\{ \k^+,\cf[\k]^\w,\cof M \} < \mathrm{cof}(\M_{\k^\w})$ for some infinite cardinal $\k$? What about $\k = \w_1$?
\end{question}

\begin{question}\label{q:aa}
Is it consistent to have $\mathrm{cof}(\M_{\k^\w}) < 2^\k$ for some infinite cardinal $\k$? What about $\mathrm{cof}(\nwd_{\k^\w}) < 2^\k$?
\end{question}

We end this section by proving a lower bound for $\mathrm{cof}(\nwd_{\k^\w})$. 
Consider the poset $\w^\k$ of functions $\k \to \w$ ordered by 
$$f \leq g \qquad \text{ if and only if } \qquad f(\a) \leq g(\a) \text{ for all } \a < \k.$$
Let $\mathrm{cof}(\w^\k,\leq)$ denote the cofinality of this poset, i.e., 
$$\mathrm{cof}(\w^\k,\leq) \,=\, \min \set{|\F|}{\forall g \in \w^\k \, \exists f \in \F \ g \leq f}.$$
This poset and its cofinality were studied by Jech and Prikry in \cite{JP}, where they proved that if $2^{\aleph_0}$ is regular and $2^{\aleph_0} < 2^{\aleph_1}$, then $\mathrm{cof}(\w^{\w_1},\leq) = 2^{\aleph_0}$ implies there is an inner model with a measurable cardinal. 
It is also known that if $\k^{\aleph_0} = \k$ then $\mathrm{cof}(\w^\k,\leq) = 2^\k$. 
It is an open problem, and has been since the appearance of Jech and Prikry's paper in 1984, whether it is consistent to have $\mathrm{cof}(\w^\k,\leq) < 2^\k$ for any uncountable $\k$. 
The following theorem shows that finding a positive answer to Question~\ref{q:aa} is at least as hard as solving this problem of Jech and Prikry. 

\begin{theorem}
$\mathrm{cof}(\w^\k,\leq) \,\leq\, \mathrm{cof}(\nwd_{\k^\w})$.
\end{theorem}
\begin{proof}
To each $F \in \nwd_{\k^\w}$ we associate a function $g_F \in \w^\k$ as follows:
$$g_F(\a) \,=\, \min \set{\mathrm{length}(s)}{ s(0) = \a \text{ and } \tr{s} \cap F = \0}.$$
Observe that $g_F$ is well-defined because $F$ is nowhere dense, which means that there must be some $s$ with $\tr{s} \sub \tr{\<\a\>} \text{ and } \tr{s} \cap F = \0$. 
Also observe that if $F \sub E$ then $g_F \leq g_E$.

If $\F \sub \nwd_{\k^\w}$ and $|\F| < \mathrm{cof}(\w^\k,\leq)$, then $\set{g_F}{F \in \F}$ is not cofinal in $(\w^\k,\leq)$, and there is some $f \in \w^\k$ such that $f \not\leq g_F$ for all $F \in \F$. Let
$$K \,=\, \set{x \in \k^\w}{x(n) = 0 \text{ if } n > f(x(0))}.$$
This set is clearly nowhere dense, because if $\tr{s}$ is a basic (cl)open subset of $\k^\w$ then we may take $t = s \cat 1 \cat 1 \cat \cdots \text{($n$ times)} \cdots \cat 1$ for some $n > f(t(0))$, and then $\tr{t} \cap K = \0$. 
However, it is also clear that $g_K \geq f$ (in fact, $g_K = f$), because if $s(0) = \a$ and $\mathrm{length}(s) < f(\a)$, then
$$x(n) \,=\, \begin{cases}
s(n) &\text{ if } n \in \mathrm{dom}(s), \\
0 &\text{ if not}
\end{cases}$$
is a point of $\tr{s} \cap F$. Hence $g_K \not\leq g_F$ for any $F \in \F$. But recall that if $F \supseteq K$ then $g_K \leq g_F$; so $F \not\supseteq K$ for any $F \in \F$. In other words, $K$ witnesses that $\F$ is not cofinal in $\nwd_{\k^\w}$.
\end{proof}

\begin{question}
Is $\mathrm{cof}(\w^\k,\leq) \,\leq\, \mathrm{cof}(\M_{\k^\w})$ for every $\k$?
\end{question}

\section{Compact Hausdorff spaces}\label{sec:compacta}

In this final section we turn our attention briefly to the class of locally compact Hausdorff spaces. 
The main theorem is an upper bound for $\mathrm{non}(\M_X)$ reminiscent of Theorem~\ref{thm:MetUnif}.

Let us point out that when $X$ is a compact Hausdorff space, the invariant $\mathrm{cov}(\M_X)$ has been widely investigate already. 
If $\BB$ is the Boolean algebra of regular open subsets of $X$, then $\mathrm{cov}(\M_X) = \mathfrak{m}(\BB)$, the \emph{Martin number} of $\BB$. The Martin number of a Boolean algebra is very important in the theory of forcing. For example, Martin's Axiom $\mathsf{MA}_{\aleph_1}$ is equivalent to the assertion that $\mathfrak{m}(\BB) > \aleph_1$ for every ccc Boolean algebra $\BB$, or equivalently, that $\mathrm{cov}(\M_X) > \aleph_1$ for every compact Hausdorff $X$ with $c(X) = \aleph_0$. 
Similarly, the assertion $\mathfrak{p} = \continuum$ is equivalent to the assertion that $\mathfrak{m}(\BB) > \aleph_1$ for every $\s$-centered Boolean algebra $\BB$, or equivalently, that $\mathrm{cov}(\M_X) > \aleph_1$ for every separable compact Hausdorff space $X$ (see \cite{Bell}). 
Furthermore, $\mathrm{cov}(\M_X) = \mathrm{cov}(\nwd_X)$ for any compact Hausdorff space $X$, and the latter has also been studied from a more topological point of view. It is sometimes called the \emph{Nov\'ak number} of $X$; see, for example, \cite{BPS}, \cite{Juhasz}, or \cite{Vel}. 

Slightly less well known is the \emph{dual Martin number} of a Boolean algebra $\BB$, denoted $\mathfrak{n}\hspace{-1.1mm}\mathfrak{v}(\BB)$ in \cite{BJ}. This is the Stone dual of $\mathrm{non}(\M_X)$, in the same way that $\mathrm{cov}(\M_X)$ is the Stone dual of $\mathfrak{m}$. That is, if $\BB$ is the Boolean algebra of regular open subsets of $X$, then $\mathrm{non}(\M_X) = \mathfrak{n}\hspace{-1.1mm}\mathfrak{v}(\BB)$. 
A kind of ``opposite'' to Martin's axiom was proved by Merrill in \cite{M}: he showed that in the Bell-Kunen model described in \cite{BK}, $\mathrm{cov}(\M_X) < \continuum$ for every separable compact Hausdorff space $X$. 
The main result of this section adds to what is known about $\mathrm{non}(\M_X)$ when $X$ is (locally) compact Hausdorff, or dually, what is known about the invariant $\mathfrak{n}\hspace{-1.1mm}\mathfrak{v}(\BB)$. 

Recall that the \emph{$\pi$-weight} of a space $X$, denoted $\pi w(X)$, is the minimum cardinality of a $\pi$-base for $X$. Let 
$$\pi w_\downarrow(X) \,=\, \min\set{\k}{\text{some nonempty open $U \sub X$ has $\pi$-weight $\k$}}.$$
We say that $X$ has \emph{uniform $\pi$-weight $\k$} if $\pi w(X) = \pi w_\downarrow(X) = \k$, or, equivalently, if every nonempty open subset of $X$ has $\pi$-weight $\k$. 

Let us note that if $X$ is metric space then $w(X) = \pi w(X)$ (see \cite[Theorem 4.1.15]{Engelking}). With that in mind, the following lemma is analogous to Lemma~\ref{lem:kw}. 

\begin{lemma}\label{lem:kw2}
Let $X$ be a compact Hausdorff space with uniform $\pi$-weight $\k$. Then there is a (not necessarily continuous) function $F: \k^\w \to X$ such that if $Z$ is a non-meager subset of $\k^\w$, then $F[Z]$ is non-meager in $X$.
\end{lemma}
\begin{proof}
Let $\B$ be a $\pi$-base for $X$ with $|\B| = \k$. 

Let us say that a sequence $\< U_0,U_1,\dots,U_\ell \>$ of sets in $\B$ is \emph{strongly decreasing} if $U_j \supseteq \closure{U_{j+1}}$ for all $j < \ell$.
Let $\mathcal T$ be the collection of all strongly decreasing finite sequences of members of $\B$. 
Given $s,t \in \mathcal T$, define $s \leq t$ to mean $s = t \rest (\mathrm{length}(s))$ (i.e., $t$ is an end-extension of $s$). 
This ordering makes $\mathcal T$ into a tree.

We claim that each $s = \< U_0,U_1,\dots,U_\ell \> \in \mathcal T$ has exactly $\k$ successors in $\mathcal T$. To see this, note that 
$$\E_s \,=\, \set{V \in \B}{\closure{V} \sub U_\ell}$$
is a $\pi$-base for $U_\ell$. Because $\pi w_\downarrow(X) = \k$, this means $|\E_s| \geq \k$, and clearly $|\E_s| \leq |\B| = \k$. Thus $|\E_s| = \k$. But $\E_s$ is precisely the set of all $V$ such that $s \cat V = \< U_0,U_1,\dots,U_\ell,V \> \in \mathcal T$, so $s$ has $\k$ successors in $\mathcal T$. 

Thus $\mathcal T$ is a subtree of $\B^{<\w}$ such that every $s \in \mathcal T$ has exactly $\k$ successors in $\mathcal T$. 
This characterizes $\mathcal T$ up to order-isomorphism: it is isomorphic to $\k^{<\w}$. 

A \emph{branch} through $\mathcal T$ is a function $b: \w \to \B$ such that $b \rest n \in \mathcal T$ for every $n$. Let $B$ denote the set of all branches through $\mathcal T$. Observe that $B$ has a natural topology, with a basis of sets of the form
$$\tr{s} \,=\, \set{b}{b \text{ is a branch through $\mathcal T$ and } b \rest (\mathrm{length}(s)) = s}.$$
Because the tree $\mathcal T$ is order-isomorphic to the tree $\k^{<\w}$, it is clear that $B$, endowed with this topology, is homeomorphic to $\k^\w$. To prove the lemma, it remains find a function $F: B \to X$ such that if $Z$ is a non-meager subset of $B$, then $F[Z]$ is non-meager in $X$. 

For each $b \in B$ we have an infinite sequence $b(0),b(1),b(2),\dots$ of members of $\B$ such that $b(n) \supseteq \closure{b(n+1)}$ for all $n$. Because $X$ is compact, this implies $\bigcap_{n \in \w} b(n) = \bigcap_{n \in \w} \closure{b(n+1)} \neq \0$. Let $F: B \to X$ be a choice function mapping each $b \in B$ to some (any) element of $\bigcap_{n \in \w} b(n)$.

We claim that if $V$ is an open dense subset of $X$, then $F^{-1}(V)$ contains an open dense subset of $B$.
To see this, 
fix some $s = \< U_0,U_1,\dots,U_\ell \> \in \mathcal T$. Because $V$ is open and dense in $X$, $V \cap U_\ell$ is a nonempty open subset of $X$. Because $\B$ is a $\pi$-base for $X$, there is some $W \in \B$ such that $\closure{W} \sub U_\ell \cap V$. Because $\closure{W} \sub U_\ell$, we have $s \cat W \in \mathcal T$. Because $W \sub V$, if $b \in \tr{s \cat W}$ then $F(b) \in b(\ell+1) = W \sub V$. 
Hence $\tr{s \cat W} \sub F^{-1}(V)$. 
Thus $\tr{s \cat W}$ is a basic (cl)open subset of $B$ such that $\tr{s \cat W} \sub \tr{s} \cap F^{-1}(V)$. 
As $s$ was arbitrary, this shows $F^{-1}(V)$ contains an open dense set.

To finish the proof, suppose $Z \sub B$ and $F[Z]$ is meager in $X$. 
This means there is a sequence $V_0,V_1,V_2,\dots$ of dense open subsets of $X$ such that $F[Z] \cap \bigcap_{n \in \w}V_n = \0$. Pulling things back through $F$, this implies $Z \cap \bigcap_{n \in \w}F^{-1}(V_n) = \0$. By the previous paragraph, $\bigcap_{n \in \w}F^{-1}(V_n)$ is a co-meager subset of $B$. Thus $Z$ is meager. 
Thus, by contraposition, if $Z$ is non-meager in $B$ then $F[Z]$ is non-meager in $X$.
\end{proof}

\begin{theorem}
Let $X$ be a locally compact Hausdorff space, and let $\k = \pi w_\downarrow(X)$. Then $\mathrm{non}(\mathcal M_X) \leq \cf[\k]^\w \cdot \non M$.
\end{theorem}

\begin{proof}
If $\k = \pi w_\downarrow(X)$ then there is a nonempty open $U \sub X$ with uniform $\pi$-weight $\k$. 
Using the fact that $X$ is locally compact, and passing to a subspace of $U$ if necessary, we may (and do) assume $\closure{U}$ is compact. 
By the previous lemma, there is a function $F: \k^\w \to \closure{U}$ such that if $Z$ is meager in $\k^\w$ then $F[Z]$ is meager in $\closure{U}$. By Theorem~\ref{thm:UW}, there is a non-meager $Z \sub \k^\w$ with $|Z| = \non M \cdot \cf[\k]^\w$. 
So $F[Z]$ is a non-meager subset of $\closure{U}$ with $|Z| = \non M \cdot \cf[\k]^\w$. 
Because $U$ is open in $X$, $F[Z]$ is also non-meager in $X$, so this shows $\mathrm{non}(\M_X) \leq \non M \cdot \cf[\k]^\w$.
\end{proof}

Because $w(X) = \pi w(X)$ and $w_\downarrow(X) = \pi w_\downarrow(X)$ whenever $X$ is a metric space (see \cite[Theorem 4.1.15]{Engelking}), this theorem essentially asserts the same thing as Theorem~\ref{thm:MetUnif}, but for locally compact spaces, and with the equality weakened to an inequality. It is worth pointing out that, as opposed to metric spaces, with compact spaces full equality cannot be established. 

\begin{theorem}
It is consistent with \zfc that there is a compact Hausdorff space $X$ with uniform $\pi$-weight $\k$ such that $\mathrm{non}(\mathcal M_X) < \cf[\k]^\w \cdot \non M$.
\end{theorem}
\begin{proof}
In \cite[Theorem 3.1.8]{BJ} it is shown that if $\BB$ is the measure algebra on $\R$ then $\mathfrak{n}\hspace{-1.1mm}\mathfrak{v}(\BB) = \non L$ (the least size of a non-Lebesgue-measurable subset of $\R$). 
Equivalently, if $X$ is the Stone dual of the measure algebra then $\mathrm{non}(\M_X) = \non L$. 
Consequently, if $\non L < \non M$ then $\mathrm{non}(\M_X) < \cf[\k]^\w \cdot \non M$. 
The consistency of $\non L < \non M$ is well known (for example, it holds in the random model). 
\end{proof}

\begin{question}
Is it consistent with \zfc that $\mathrm{non}(\mathcal M_X) = \cf[\k]^\w \cdot \non M$ for every compact Hausdorff space $X$ with uniform $\pi$-weight $\k$?
\end{question}


\begin{thebibliography}{99}

\bibitem{BPS} B. Balcar, J. Pelant, and P. Simon, ``The space of ultrafilters on $N$ covered by nowhere dense sets,'' \emph{Fundamenta Mathematicae} \textbf{110} (1980), pp. 11--24.
\bibitem{BJ} T. Bartoszy\'nski and H. Judah, \emph{Set Theory: On the Structure of the Real Line}, A K Peters (1995).
\bibitem{BT} J. Baumgartner and A. Taylor, ``Saturation properties of ideals in generic extensions. I," \emph{Transactions of the American Mathematics Society} \textbf{270} (1982), pp. 557--574. 
\bibitem{Bell} M. Bell, ``On the combinatorial principle $P(\continuum)$,'' \emph{Fundamenta Mathematicae} \textbf{114} (1981), pp. 149--157.
\bibitem{BK} M. Bell and K. Kunen, ``On the PI character of ultrafilters,'' \emph{La Soci\'et\'e Royale du Canada, L'Acad\'emie des Sciences, Comptes Rendues Math\'ematiques (Mathematical Reports)} \textbf{3} (1981), pp. 351--356. 
\bibitem{Brian} W. Brian, ``The Borel partition spectrum at successors of singular cardinals,'' to appear in \emph{Proceedings of the American Mathematical Society}. 
\bibitem{Engelking} R. Engelking, \emph{General Topology}. Sigma Series in Pure Mathematics, 6, Heldermann, Berlin (revised edition), 1989.
\bibitem{ET} P. Erd\H{o}s and A. Tarski, ``On families of mutually exclusive sets,'' \emph{Annals of Mathematics (2)} \textbf{44} (1943), pp. 315--329.
\bibitem{Fremlin} D. Fremlin, ``The partially ordered sets of measure theory and Tukey's ordering,'' \emph{Note di Matematica} \textbf{11} (1991), pp. 177--214.
\bibitem{Jech1} T. Jech, ``The closed unbounded filter over $P_\k(\lambda)$,'' \emph{Notices of the American Mathematical Society} \textbf{18} (1971), p. 663, abstract. 
\bibitem{Jech2} T. Jech, ``Some combinatorial problems concerning uncountable cardinals,'' \emph{Annals of Mathematical Logic} \textbf{5} (1973), pp. 165--198.
\bibitem{JP} T. Jech and K. Prikry, ``Cofinality of the partial ordering of functions from $\Omega_1$ into $\Omega$ under eventual domination,'' \emph{Mathematical Proceedings of the Cambridge Philosophical Society} \textbf{95} (1984), pp. 25--32.
\bibitem{Juhasz} I. Juh\'asz and W. Weiss, ``Nowhere dense choices and $\pi$-weight,'' \emph{Annales Mathematicae Silesianae} \textbf{2} (1986), pp. 85-91.
\bibitem{JMPS} W. Just, A. R. D. Mathias, K. Prikry, and P. Simon, ``On the existence of large $p$-ideals,'' \emph{Journal of Symbolic Logic} \textbf{55} (1990), pp. 457--465.
\bibitem{Kechris} A. Kechris, Classical Descriptive Set Theory, Graduate Texts in Mathematics, vol.
156, Springer-Verlag, 1995.
\bibitem{Matet} P. Matet, ``Large cardinals and covering numbers,'' \emph{Fundamenta Mathematicae} \textbf{205} (2009), pp. 45--75.
\bibitem{M} J. Merrill, ``A class of consistent anti-Martin's axioms,'' \emph{Pacific Journal of Mathematics} \textbf{143} (1990), pp. 301--312.
\bibitem{Shelah1} S. Shelah, ``Advances in cardinal arithmetic," In: Sauer, N.W., Woodrow, R.E., Sands, B. (eds.) \emph{Finite and Infinite Combinatorics in Sets and Logic} (1993). NATO ASI Series, vol 411. Springer, Dordrecht.
\bibitem{Shelah2} S. Shelah, \emph{Cardinal Arithmetic}, Oxford Logic Guides, 29, Oxford University Press, New York, 1994.
\bibitem{Shioya} M. Shioya, ``A proof of Shelah's strong covering theorem for $\mathcal P_\k \lambda$," \emph{Asian Journal of Mathematics} \textbf{12} (2008), pp. 83--98.
\bibitem{Stone} A. H. Stone, ``Non-separable Borel sets,'' \emph{Dissertationes Mathematicae} \textbf{28} (1962), pp. 1--40.
\bibitem{Vel} B. Veli\v{c}kovi\'c, ``Jensen's $\square$ principles and the Nov\'ak number of partially ordered sets,'' \emph{Journal of Symbolic Logic} \textbf{51} (1986), pp. 47--58.

\end{thebibliography}
\end{document}